\documentclass[letterpaper, 10 pt, conference]{ieeeconf}  
\usepackage{amsfonts}
\usepackage{amsmath}
\usepackage{ntheorem}
\usepackage[T1]{fontenc}
\usepackage[utf8]{inputenc}
\usepackage{xcolor}
\usepackage{bm}
\newtheorem{theorem}{Theorem}
\newtheorem{lemma}[theorem]{Lemma}

\newtheorem{remark}[theorem]{Remark}

\IEEEoverridecommandlockouts                              

\usepackage{graphics} 
\usepackage{epsfig} 
\usepackage{amsmath} 
\usepackage{amssymb}  

\title{\LARGE \bf  Delay is Necessary for a  Potential to Achieve Exponential Stabilization  of the Wave Equation via Internal Control}

\author{Cr\'edo Fanou$^{1}$, \quad  Ka\"{\i}s Ammari$^{2}$,\quad Islam Boussaada$^{1,3,*}$
\thanks{*This work is part of the SPECTRE-EDP project, funded by the CNRS-AFRICA Joint Research Program.}
\thanks{$^{1}$Universit\'e Paris-Saclay, CNRS, CentraleSup\'elec, Inria, Laboratoire des Signaux et Syst\`emes (L2S), Gif-sur-Yvette, France.}%
\thanks{$^{*}$Correspondint author; e-mail: \tt{islam.boussaada@ipsa.fr}}
\thanks{$^{2}$LR Analyse et Contr\^{o}le des EDPs, LR 22ES03,  Universit\'{e} de Monastir, 5019 Monastir, Tunisie.}
\thanks{$^{3}$Institut Polytechnique des Sciences Avanc\'ees (IPSA), France.}
}

\begin{document}

\maketitle
\thispagestyle{empty}
\pagestyle{empty}


\begin{abstract}
In this work, we study the stabilization of the wave equation using an internal delayed potential. Interestingly, the stabilization mechanism is entirely induced by the delay, since exponential stabilization cannot be achieved in its absence. We first prove the well-posedness of the associated initial–boundary value problem. Then, thanks to the parametric analysis of the corresponding  quasipolynomial, we design a delayed potential feedback law which, together with appropriate initial conditions, ensures the exponential decay rate for the resulting closed-loop system. The control of the  transverse vibration of a string illustrates the effectiveness of the result.
\end{abstract}

\section{INTRODUCTION}
The internal control and stabilization of the wave equation have attracted sustained attention over the past decades, both for their theoretical significance and for their relevance to applications in structural mechanics, acoustics, and wave propagation. Foundational contributions such as \cite{BLR1992} established sharp geometric conditions ensuring observability and boundary stabilization of wave equations. Also works such as \cite{Russell} , as well as \cite{CoxZuazua}, further clarified decay mechanisms and indirect damping strategies. These results highlight how the interplay between geometry, damping mechanisms, and spectral properties governs the asymptotic behavior of wave dynamics.
While boundary damping has been extensively studied, internal control mechanisms offer greater flexibility and may better capture localized physical effects. However, it is well known—see, for instance, the delay-sensitivity analyses of Richard Datko\cite{datko1986example}—that the structure of the feedback law critically influences stability properties. In particular, an instantaneous internal potential cannot achieve exponential stabilization. The spectrum of the associated generator cannot be shifted sufficiently to the left half-plane, preventing uniform exponential decay.
This limitation naturally motivates the introduction of delayed internal potentials. Rather than being merely an obstacle, the delay may fundamentally reshape the spectral structure of the system. Properly designed delayed feedback can modify the characteristic equation in a way that enables spectral shifting otherwise impossible in the instantaneous case. In this perspective, delay becomes a constructive mechanism for stabilization rather than a source of instability.

By combining spectral analysis and quasipolynomial techniques we aim to demonstrate that delay, when properly structured, provides a powerful mechanism for exponential stabilization.

More precisely, the system under consideration  is given by:
\begin{equation}
\left\{
\begin{array}{ll}
u_{tt}(x,t)  - u_{xx}(x,t) + \mathbf{u}(x,t) =0, x\in (0,\ell),\ t>0,\\
 \text{where } \mathbf{u}(x,t)= \alpha\, u(x,t-\tau),\\
u(0,t) = 0 = u(\ell,t), t> 0, \\
u(x,0)=u_{0}(x), u_t(x,0) = u_1(x), x\in (0,\ell), \\
u(x,t-\tau) = f_0(t-\tau), x \in(0,\ell), t \in (0,\tau).
\end{array}
\right.  \label{ondes}
\end{equation}
The constant $\tau >0$ denotes the time-delay and $\ell >0$,
$\alpha,$ are real numbers such that $\alpha\neq 0$ and the initial data $u_{0},u_1$ and $f_0$ are given functions belonging to suitable spaces that will be precised later.

The contribution of the paper is twofold: first it aims to show the positive effect of the delay in internal potential stabilization. In fact, without delay no exponential stabilization can be achieved and in the presence of delay, the latter has to be appropriately chosen since a delay-independent stabilization is not possible.  Secondly, inspired from \cite{ammari2024prescribing, ammari2026multiplicity}, its objective is to determine appropriate initial conditions that guarantee the exponential stabilization of the wave equation via a delayed potential.

The remaining paper is organized as follows. In Section \ref{sec2} for the sake of self-containment, the well-posedness of the considered problem is investigated. Next, Section \ref{sec4} presents a spectral analysis  where the domain of parameters $(\tau,\alpha)$ ensuring a suitable spectrum distribution in the the complex left half-plane is established. Section \ref{sec5} provides the main results of this paper, focusing on the study of delay-independent and delay-dependent stability. Finally, to demonstrate the effectiveness and practical relevance of the proposed approach, Section \ref{simm} applies the results to the control of transverse vibrations in a string in order to illustrate the main results before concluding  in Section \ref{conc}.
\color{black}
\section{Well-posedness of problem}\label{sec2}
In this section, we first transform the delay terms by adding  new unknowns. Then, through a shifted operator approach, we use semigroup theory and the Lumer-Phillips theorem to prove the existence and uniqueness of the solution of problem \eqref{ondes}. To do so, let us introduce as in \cite{nicaise2006stability}:
\begin{equation*}\label{changev}
    z(x,\rho,t)=u(x,t-\tau\rho),\quad x\in (0,\ell),\; \rho\in (0,1),\; t>0.
\end{equation*}
Hence, problem (\ref{ondes}) is equivalent to:
\begin{equation}\label{NP}
    \begin{cases}
    u_{tt}(x,t)  - u_{xx}(x,t) +  \alpha z(x,1,t) =0,\\
    \tau z_t(x,\rho,t) + z_\rho(x,\rho,t)=0,\\
    z(x,0,t)=u(x,t),\quad x\in (0,\ell),\\
    z(x,\rho,0)= f_0(-\tau \rho) := z_0(x,\rho),\\
    u(0,t) = 0 = u(\ell,t),\\ 
    u(x,0)=u_{0}(x), u_t(x,0) = u_1(x),
    \end{cases}
\end{equation}
Now, we define the energy space: 
\begin{equation}
    \mathcal{H} := H_0^1(0,\ell)\times L^2(0,\ell)\times L^2((0,\ell)\times(0,1)),
\end{equation}
which is a Hilbert space equipped with inner product:
\begin{equation*}
    \langle V_1,V_2\rangle_\mathcal{H}=\int_0^\ell u_x^1 u_x^2dx+\int_0^\ell v^1 v^2dx+\\ \xi\int_0^\ell\int_0^1z^1z^2d\rho dx
\end{equation*}
for $V_1=(u^1,v^2,z^1),\, V_2=(u^2,v^2,z^2)$ and $\xi>0$ nonnegative real numbers to be defined later.\\
Let $\mathcal{A}$ defined as follows:
\begin{equation}\label{abst}
 \mathcal{A} \begin{pmatrix}
     u\\
     v\\
     z
 \end{pmatrix}=\begin{pmatrix}
     v\\
    u_{xx}  - \alpha z(\cdot,1) \\
    -\tau^{-1}z_\rho
\end{pmatrix}
\end{equation}
with domain $\mathcal{D}(\mathcal{A})$ defined as the set of $V: =(u,u_t,z)^T$ satisfying:
\begin{align}
 (u,v,z)^T & \in  \big(H_0^1(0,\ell)\cap H^2(0,\ell)\big)\times H_0^1(0,\ell)\times \notag \\
 & \hspace{2.7cm} L^2((0,\ell);H^1(0,1)),\label{dom1}\\
 & \hspace{2.7cm} z(\cdot,0)=u \text{ in } (0,\ell).\label{dom2}
\end{align}
Therefore, if $V_0,\, V\in \mathcal{H}$ , the problem (\ref{NP}) is formally equivalent to the following abstract evolution equation in the Hilbert space $\mathcal{H}$: 

\begin{equation}\label{EV}
    \begin{cases}
        V'(t)=\mathcal{A} V(t), & t>0 \\
        
        V(0)=V_0
    \end{cases}
\end{equation}
with $V_0 :=(u_0,u_1,z_0)$. Under the condition $ \xi=2|\alpha|\tau,$ we introduce the shifted operator $\mathcal{A}_{sh}= \mathcal{A}-\gamma I,$ where $\gamma=\frac{|\alpha|}{2}\max\, \big(1, \ell^2\big)$. Using the Lumer-Phillips theorem, it follows that $\mathcal{A}_{sh}$ generates a $C_{0}$ semigroup of contraction on $\mathcal{H}$, which is denoted by $(e^{t\mathcal{A}_{sh}})_{t \geq 0}$.\\
Then, the well-posedness of problem \eqref{NP} is ensured by:
\begin{theorem}
Let $V_0\in \mathcal{H}$, then there exists a unique solution $V\in \mathcal{C}(\mathbb{R}_{+};\mathcal{H})$ of  problem \eqref{EV}. In particular, if $V_{0}\in \mathcal{D}(\mathcal{A})$ we get:
\begin{equation*}
    V\in \mathcal{C}(\mathbb{R}_{+};\mathcal{D}(\mathcal{A}))\cap \mathcal{C}^{1}(\mathbb{R}_{+};\mathcal{H}).
\end{equation*}
\end{theorem}
\begin{proof} Since $\mathcal{A}_{sh}$ generates a $C_{0}$ semigroup of contraction on $\mathcal{H}$, following  \cite{pazy2012semigroups, engel2000one}, we deduce that $\mathcal{A}$ generates a $C_0$ semigroup $(e^{t\mathcal{A}})_{t \geq 0}$.
\end{proof}
\begin{remark}\label{rem1} Denoting by $\rho(\mathcal{A})$ and $R_{\mathcal{A}}(\lambda)$ respectively the resolvent set and  the resolvent operator  of $\mathcal{A}$, note that $R_{\mathcal{A}}(\lambda)$ is continuous and compact. Therefore, we  have:  $\sigma(\mathcal{A})=\sigma_{p}(\mathcal{A}),$  where $\sigma(\mathcal{A})$ is the spectrum of $\mathcal{A}$ and $\sigma_{p}(\mathcal{A})$ the point spectrum. 
\end{remark}

\section{ Spectral analysis}\label{sec4}

As stated in \ref{rem1}, the spectrum consists of eigenvalues of $\mathcal{A}$. We now provide the following characterization of the eigenvalues and eigenvectors of $\mathcal{A}$.
\begin{lemma}
    A complex number $\lambda\in\mathbb{C}$ is an eigenvalue of $\mathcal{A}$ if, and only if,
    \begin{equation}
     \lambda^{2} + \frac{n^2\pi^2}{\ell^2} + \alpha e^{-\lambda\tau} = 0,\quad n\in\mathbb{N}^{*}.
\end{equation}
where the corresponding eigenvector $F_{\lambda}=(u_{\lambda},v_{\lambda},z_{\lambda})$ is given by:
\begin{equation}\label{fun}
    \begin{cases}
        u_{\lambda}(x)=\sin\big(\frac{n \pi}{\ell}x\big),\, x\in (0,\ell),\\
        v_{\lambda}(x)=\lambda \sin\big(\frac{n \pi}{\ell}x\big),\, x\in (0,\ell),\\
        z_{\lambda}(x,\rho)= \sin\big(\frac{n \pi}{\ell}x\big)e^{-\lambda\tau\rho},\, \\
        \qquad \hspace{0.7cm} (x,\rho)\in(0,\ell) \times (0,1).
    \end{cases}
\end{equation}
\end{lemma}

Now, we want to study how the spectrum is located relative to the imaginary axis. For this purpose, let us set 
\begin{equation}\label{QuasiGen}
    Q(\lambda,\tau)=  \lambda^{2} + \frac{n^2\pi^2}{\ell^2} + \alpha e^{-\lambda\tau}.
\end{equation}
\begin{remark}\label{RemDescate}
   Notice that when the delay is set to zero, then the characteristic function \eqref{QuasiGen} reduces to a sparse second order polynomial and it is immediate following Decartes' rule that the tuning of $\alpha$ cannot allow to a spectrum with only negative real parts.   
\end{remark} 
\noindent By changing variables: $z=\lambda\tau$, 
the quasipolynomial \eqref{QuasiGen} reduces to
\begin{equation}
    \tilde{Q}(z)=z^2+\tilde{\beta} + \tilde{\alpha}e^{-z}
\end{equation}
where $\tilde{\beta}=\tau ^2\frac{n^2\pi^2}{\ell^2},\; \tilde{\alpha}=\tau ^2\alpha.$\\
Then, the following lemma gives a necessary and sufficient condition in the parameter plane  to have these roots in the complex left half-plane.

\begin{lemma}
  \label{atay} \cite{atay1999balancing} For $n\in\mathbb{N}^{*}$ fixed.\\
    The roots of $\tilde{Q}$ are in the complex left half-plane if and only if:
    \begin{equation}\label{ine}
        0<(-1)^{k+1} \tilde{\alpha}<\min \big\{ \tilde{\beta} - k^2\pi^2,(k+1)^2\pi^2-\tilde{\beta} \big\}
    \end{equation}
    for some nonnegative integer $k$.  
\end{lemma} 

Lemma \ref{atay} is proved by carrying out an analysis based on the parameters $(\tilde{\beta}, \tilde{\alpha})$ that will consist of using the continuity of the roots as function of parameters. Hence, the parameters plane $(\tilde{\beta}, \tilde{\alpha})$ can be partitioned into regions inside which the number of roots in the right half-plane remains constant \cite{michiels2014stability} and this number is known thinks to Stepan-Hassard \cite{Stepan1989Retarded, huang1985characteristic} formula as illustrate in Fig \ref{fig:StabilityDomain}. number of roots in the right half-plane in relation to the parameter plane.
\begin{figure}[h]
    \centering
    {\includegraphics[trim={10mm 27mm 22mm 4mm},clip,width=1\linewidth]{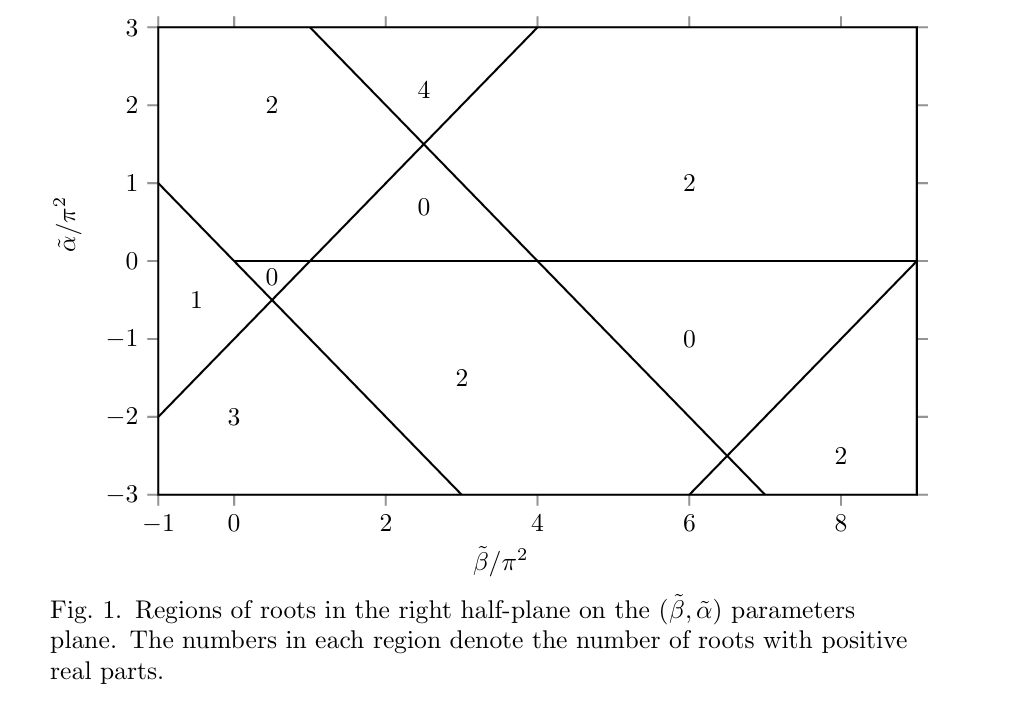}} \vskip-0.5\baselineskip
    \caption{Regions of roots in the right half plane on the $\left(\tilde{\beta},\tilde{\alpha}\right)$ parameters plane. The numbers in each region denote the number of roots with positive real parts.\label{fig:StabilityDomain}}
\end{figure}
\begin{remark}\label{depen}
 Using the change of variable on \eqref{QuasiGen}, condition \eqref{ine} on $\alpha$ and $\tau$
implies that, for $n\in\mathbb{N}^*$,
\begin{multline}\label{line}
  0<(-1)^{k+1}\ell^2\tau^2 \alpha<\\
  \min \big\{ n^2\pi^2\tau^2 - k^2\ell^2\pi^2,\,(k+1)^2\ell^2\pi^2-n^2\pi^2\tau^2 \big\}  
\end{multline}
for some nonnegative integer $k$.   
\end{remark}
\begin{remark}
    It is worth noting that condition \eqref{line} can be interpreted as follows.
    If $\alpha$ is negative (respectively positive) then $k$ has to be odd (respectively even), this is obtained from first inequality.
     Now from the second inequality, the minimum $ \min \big\{ n^2\pi^2\tau^2 - k^2\ell^2\pi^2,\,(k+1)^2\ell^2\pi^2-n^2\pi^2\tau^2 \big\} $ has to be positive, so both quantities need to be positive.

     Furthermore, straightforward algebraic  manipulations allow for the following simpler inequality $$\frac{n\tau}{\ell}-1<k<\frac{n\tau}{\ell}$$. So that, if it exists,  $k$ needs to be exactly the integer part of $\frac{n\tau}{\ell}$. In particular, if $\frac{n\tau}{\ell}\in\mathbb{N}^*$, then the existence of such a $k$ is not possible, meaning that some of zeros of $Q$ are necessarily with non negative real parts.
\end{remark}

\section{Main results}\label{sec5}
In this section we first provide first a delay-independent stability analysis, then the investigation of the delay-dependent stabilization is considered.

\subsection{Delay-independent stability cannot occur}\label{DI}
Consider the family of quasipolynomials $(Q_n)_{n\in\mathbb{N}^*}$ depending on the parameters $(\ell,\alpha,\tau) \in  \mathbb{R}^{*}_{+}\times\mathbb{R}^{*}\times\mathbb{R}_{+}$:
\begin{equation}
    Q_n(s;\ell,\alpha,\tau)=s^2+\frac{n^2\pi^2}{\ell^2}+\alpha e^{-s\tau}
\end{equation}
The delay-independent stability property is characterized by uniform stability independently of  delay with the absence of crossing frequencies $\omega\in\mathbb{R}$, that is, $Q_n(i\omega)\neq0$ (see \cite{michiels2014stability} for more details).
\begin{theorem}\label{independ}
    Considering the solution of problem \eqref{ondes}, uniform stability independent of the delay cannot be achieved.
\end{theorem}
\begin{remark}
 In particular, stabilization cannot be ensured through an instantaneous potential of the form $\mathbf{u}(x,t)=\alpha u(x,t).$  
\end{remark}

\begin{proof}[Proof of Theorem \ref{independ}]. It suffices to show the existence of crossing frequencies corresponding to critical delays, which reduces to studying $Q_n(i\omega)=0$. This leads to
\begin{equation}
    \omega^4-\frac{2n^2\pi^2}{\ell^2}\omega^2 - \big(\alpha^2-\frac{n^4\pi^4}{\ell^4}\big)=0.
\end{equation}
 The following configurations are then considered.
 \begin{enumerate}
\item [(a)]  If $\alpha^2\geq\frac{n^4\pi^4}{\ell^4}$, $Q_n$ admits a unique  pair of complex conjugate root $\pm i\omega_+$ on the imaginary axis with $\tau_+$
the critical delay associated.
\item[(b)] If $\alpha^2<\frac{n^4\pi^4}{\ell^4}$, $Q_n$  admits  two pair of complex conjugate roots on the imaginary axis that we denote $\pm i\omega_+$ and $\pm i\omega_-$ such as $0<\omega_-<\omega_+$. The critical delays associated are respectively denote as $\tau_-$ and $\tau_+$.
\end{enumerate}
As a result, there always exists at least one pair of roots on the imaginary axis. Moreover, as discussed in Remark \ref{RemDescate} and  confirmed by Remark \ref{depen} when the delay is set to zero, the applied control cannot stabilize.
\end{proof}

The delay therefore acts as an important control parameter in the stabilization process.

\subsection{delay-dependent stabilization}
In this part, the control parameters $\alpha$ and $\tau$ are considered under condition \eqref{line}.
Then the eigenvalues of $\mathcal{A}$ is localized in the complex left half-plane: 
\begin{align}
   \sigma(\mathcal{A}) = & \underset{n\geq 1}{\cup}  \sigma_{n}(\mathcal{A}) \notag \\
    =& \underset{n\geq 1}{\cup} \big\{ \lambda \in \mathbb{C}_-\, :\, \lambda^{2} + \frac{n^2\pi^2}{\ell^2} + \alpha e^{-\lambda\tau} = 0 \big\}.
\end{align}

The next theorem  provides necessary and sufficient conditions that guaranty the  exponential decay of the closed-loop system's \eqref{ondes} solution  

\begin{theorem}\label{dec}
    Under the initial conditions
    \begin{equation}\label{ini}
    u_0(x)=\zeta_0\sin\big(\frac{n \pi}{\ell}x\big),\quad u_1(x)=\zeta_1\sin\big(\frac{n \pi}{\ell}x\big)
    \end{equation}
    with $(\zeta_0,\zeta_1,n)\in\mathbb{R}\times\mathbb{R}\times\mathbb{N}^{*}$ and for $(\tau,\alpha)\in\mathbb{R}_{+}^{*}\times\mathbb{R}^{*},$ satisfying \eqref{line}, the solution of \eqref{ondes} decays exponentially. 
\end{theorem}


\begin{remark}
The quasimode-type initial conditions \eqref{ini} allow the selection of a specific mode and the projection of the solution of \eqref{ondes} onto it prior to the application of the stabilization mechanism. As a result, the stabilization strongly depends on the prescribed initial data. This feature is particularly relevant in engineering applications, where suitably designed initial data may be used to achieve stabilization around the zero solution.
\end{remark}
\begin{proof}[Proof of Theorem \ref{dec}]
With the initial conditions (\ref{ini}), the only relevant quasipolynomial is:
\begin{equation*}
 Q_n(\lambda,\tau)= \lambda^{2} + \frac{n^2\pi^2}{\ell^2} + \alpha e^{-\lambda\tau}.  
\end{equation*}
 Then, by taking $\tau$ and $\alpha$ satisfying condition \eqref{line} from Remark \ref{depen}, the results follows directly from Lemma \ref{atay}. 
\end{proof}

The above analysis shows that, when using a potential-type control, the delay, and more precisely its appropriate choice, is necessary for stabilization.

\section{Simulations}\label{simm}
We present simulation results applying an internal delayed potential to the transverse vibration of a string of length $l>0$ whose dynamics are modeled by the following equation
\begin{equation}\label{ondesnum}
   \begin{cases}
        U_{\mathbf{tt}}(\mathbf{x,t})-c^2U_{\mathbf{xx}}(\mathbf{x,t})=-V(\mathbf{x,t}),\,\\ 
        \qquad \hspace{1cm} (\mathbf{x,t})\in(0,l)\times(0,+\infty),\\
        U(0,t)=0,\quad U(l,t)=0, \qquad t>0,\\
        U(\mathbf{x},0)=f(\mathbf{x}),\quad  U_{\mathbf{t}}(\mathbf{x,0})=g(\mathbf{x}), \, \mathbf{x}\in(0,l),
   \end{cases}
\end{equation}
where
\begin{equation*}
  U(\mathbf{x,t}) := u(x,t),\; V(\mathbf{x,t)}:=\begin{cases} \label{contnum}
     \frac{\alpha}{d^2}  U(\mathbf{x,t}-\bm{\tau}),\\ (\mathbf{x,t})\in(0,l)\times(0,+\infty),\\
      \\
  0, \; \mathbf{t}\in(0,\bm{\tau}),
\end{cases}
\end{equation*}
with
\begin{equation*}
     \mathbf{x}=l\frac{x}{\ell}, \quad \mathbf{t}=d \, t\;\; \Big(d:=\frac{l}{c\,\ell}\Big),\quad\bm{\tau}=d\, \tau,
\end{equation*}
and $\tau,\; \alpha$ satisfying condition in Theorem \ref{dec}.\\
\indent  The problem is discretized using the finite difference method. Now to show the efficiency of the control considered, we  chose the following initial conditions
\begin{equation}\label{mdes}
    f(x)=A\sin\Big(\frac{n\pi}{l}x\Big),\quad g(x)=0,
\end{equation}
Let us denote by $\Delta x$, $\Delta t$ and $T_f$ respectively the space step, the time step and the final simulation time such that: $\Delta x=0.05,\, \Delta t=0.005,\, T_f=100$  for $l=10\, ,\, c=1.118 \,,\,d=8.9443\, \,,\, A=1\,, \, n=1.$ Fig (A) illustrates the behavior of the solution without control.
\begin{figure}[h]
    \centering
    \includegraphics[width=0.985\linewidth]{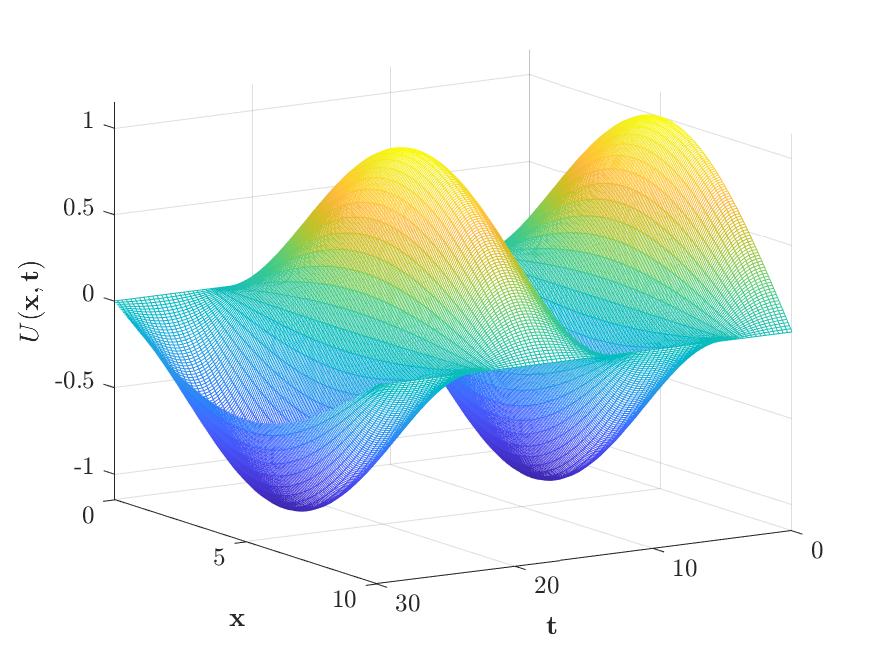}\caption{Without control. \label{fig:WithoutControl}}
\end{figure}

Fig \ref{fig:ClosedLoopCase1}, Fig \ref{fig:ClosedLoopCase2} and Fig \ref{fig:ClosedLoopCase3} illustrate the behavior of the closed-loop system under the control obtained by choosing a set of parameters $(\tau,\alpha)$ satisfying condition \eqref{line}, in the following cases: case 1 correspond to $\tau=\frac{3}{2}$,  $\alpha=5$ then $k=1$; case 2 correspond to $\tau=\frac{3}{2}$,  $\alpha=3$ then $k=1$ and case 2 correspond to $\tau=\frac{5}{2}$,  $\alpha=-1.7766$ with $k=2$.  
\begin{figure}[h]
    \centering
    \includegraphics[width=0.985\linewidth]{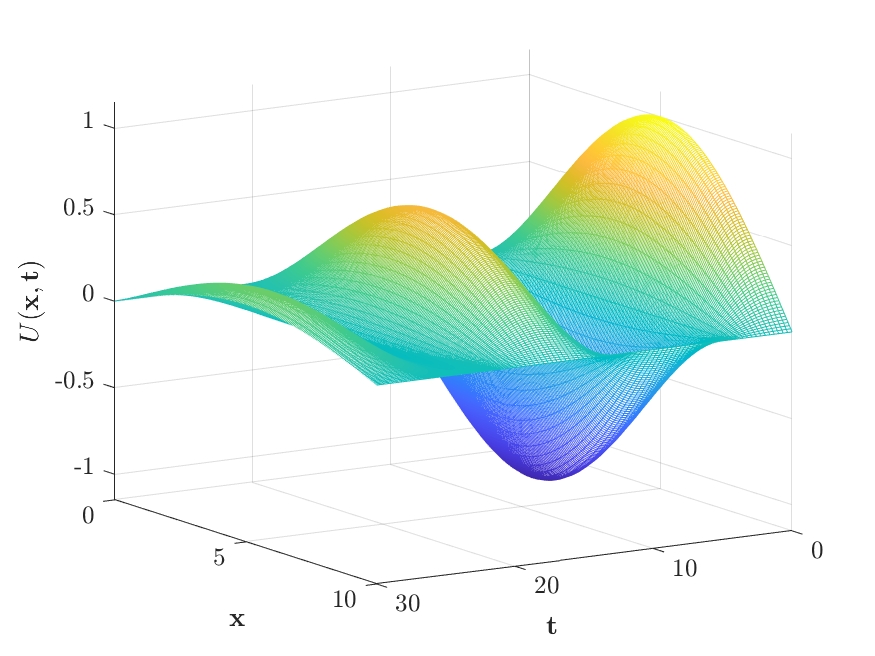}\caption{Closed-loop: case 1 ($\tau=\frac{3}{2}$,  $\alpha=5$, $k=1$). \label{fig:ClosedLoopCase1}}
\end{figure}
 
\begin{figure}[h]
    \centering
    \includegraphics[width=0.985\linewidth]{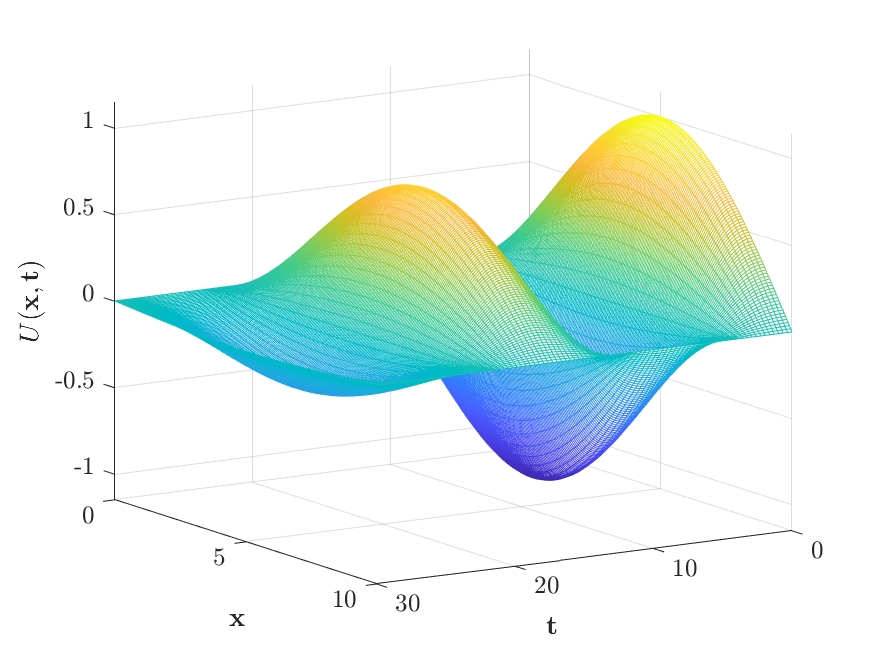}\caption{Closed-loop : case 2 ($\tau=\frac{3}{2}$, $\alpha=3$, $k=1$). \label{fig:ClosedLoopCase2}}
\end{figure}

\begin{figure}[h!]
    \centering
    \includegraphics[width=0.985\linewidth]{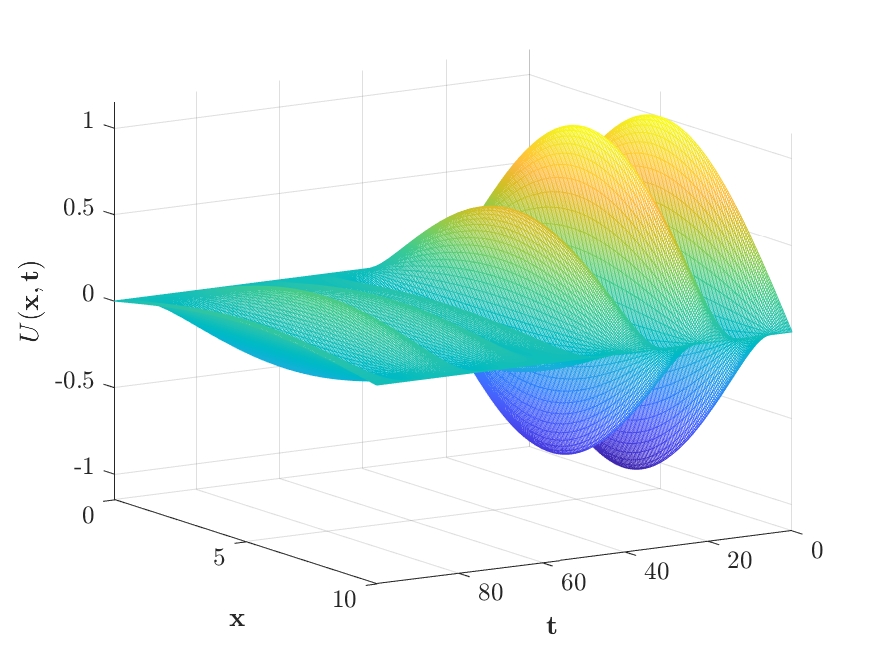}\caption{
     Closed-loop: case 3 ($\tau=\frac{5}{2}$, $\alpha=-1.7766$,   $k=2$).\label{fig:ClosedLoopCase3}}
\end{figure}
We observe that the decay is fastest when the parameters $(\tau,\alpha)$ correspond to Case 2, followed by Case 3 and finally Case 1. This is confirmed by Fig \ref{fig:placeholder}, which shows a numerical representation of the spectrum in each of these three cases.
\begin{figure}[h!]
    \centering
    {\includegraphics[width=0.95\linewidth]{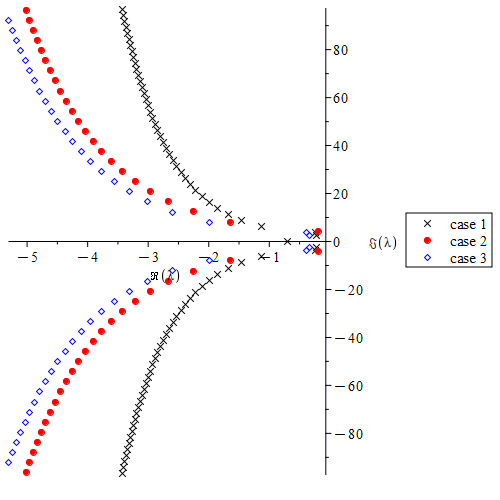}}
    \caption{Closed-loop spectrum for each case.}
    \label{fig:placeholder}
\end{figure}

\section{Concluding Remarks}\label{conc}
In conclusion, we have investigated the stabilization of the wave equation through an internal delayed potential. A distinctive feature of this approach is that the stabilization mechanism is entirely induced by the delay, since exponential stabilization cannot be achieved in its absence. After establishing the well-posedness of the associated initial–boundary value problem, we performed a parametric analysis of the corresponding quasipolynomial. This analysis enabled the design of a delayed potential feedback law which, combined with suitable initial conditions, guarantees an exponential decay  for the resulting closed-loop system.

\medskip
\section*{ACKNOWLEDGMENT}
We warmly thank our colleagues {\sc Silviu Niculescu} (L2S, University Paris-
Saclay) for valuable discussions on zeros distribution of quasipolynomials and {\sc Sami
Tliba} (L2S, CNRS) for his great help in developing the finite difference method numerical simulations. Last but not least, we thank {\sc Karim Trabelsi} (IPSA Paris)
for careful reading of the manuscript and for comments.


\bibliographystyle{plain}
\bibliography{biblio.bib}

\begin{thebibliography}{10}

\bibitem{ammari2024prescribing}
Ka{\"\i}s Ammari, Islam Boussaada, Silviu-iulian Niculescu, and Sami Tliba.
\newblock Prescribing transport equation solution's decay via multiplicity
  manifold and autoregressive boundary control.
\newblock {\em International Journal of Robust and Nonlinear Control},
  34(10):6721--6740, 2024.

\bibitem{ammari2026multiplicity}
Ka{\"\i}s Ammari, Islam Boussaada, Silviu-Iulian Niculescu, and Sami Tliba.
\newblock Multiplicity manifolds as an opening to prescribe exponential decay:
  auto-regressive boundary feedback in wave equation stabilization.
\newblock {\em Acta Applicandae Mathematicae}, 201(1):3, 2026.

\bibitem{atay1999balancing}
Fatihcan~M Atay.
\newblock Balancing the inverted pendulum using position feedback.
\newblock {\em Applied Mathematics Letters}, 12(5):51--56, 1999.

\bibitem{BLR1992}
Claude Bardos, Gilles Lebeau, and Jeffrey Rauch.
\newblock Sharp sufficient conditions for the observation, control, and
  stabilization of waves from the boundary.
\newblock {\em SIAM Journal on Control and Optimization}, 30(5):1024--1065,
  1992.

\bibitem{CoxZuazua}
Steven Cox and Enrique Zuazua.
\newblock The rate at which energy decays in a damped string.
\newblock {\em Communications in Partial Differential Equations},
  19(1-2):213--243, 1994.

\bibitem{datko1986example}
Richard Datko, John Lagnese, and MP818942 Polis.
\newblock An example on the effect of time delays in boundary feedback
  stabilization of wave equations.
\newblock {\em SIAM journal on control and optimization}, 24(1):152--156, 1986.

\bibitem{engel2000one}
Klaus-Jochen Engel and Rainer Nagel.
\newblock {\em One-parameter semigroups for linear evolution equations}.
\newblock Springer, 2000.

\bibitem{huang1985characteristic}
Falun Huang.
\newblock Characteristic conditions for exponential stability of linear
  dynamical systems in hilbert spaces.
\newblock {\em Ann. of Diff. Eqs.}, 1:43--56, 1985.

\bibitem{michiels2014stability}
Wim Michiels and Silviu-Iulian Niculescu.
\newblock {\em Stability, control, and computation for time-delay systems: an
  eigenvalue-based approach}.
\newblock SIAM, 2014.

\bibitem{nicaise2006stability}
Serge Nicaise and Cristina Pignotti.
\newblock Stability and instability results of the wave equation with a delay
  term in the boundary or internal feedbacks.
\newblock {\em SIAM Journal on Control and Optimization}, 45(5):1561--1585,
  2006.

\bibitem{pazy2012semigroups}
Amnon Pazy.
\newblock {\em Semigroups of linear operators and applications to partial
  differential equations}.
\newblock Springer Science \& Business Media, 2012.

\bibitem{Russell}
David~L. Russell.
\newblock Controllability and stabilizability theory for linear partial
  differential equations: Recent progress and open questions.
\newblock {\em SIAM Review}, 20(4):639--739, 1978.

\bibitem{Stepan1989Retarded}
G\'abor St\'{e}p\'{a}n.
\newblock {\em Retarded dynamical systems: stability and characteristic
  functions}, volume 210 of {\em Pitman Research Notes in Mathematics Series}.
\newblock Longman Scientific \& Technical, Harlow, 1989.

\end{thebibliography}

\end{document}